\documentclass[a4paper,leqno,12pt,twoside]{amsart}

\usepackage[utf8]{inputenc}
\usepackage[english]{babel}
\usepackage{amsfonts}
\usepackage{amsmath}
\usepackage{mathrsfs}
\usepackage{amssymb}
\usepackage{amsthm}
\usepackage{amscd}
\usepackage{latexsym}
\usepackage{amstext}
\usepackage{amsxtra}
\usepackage[alphabetic,backrefs,lite]{amsrefs} 
\usepackage{color}
\usepackage{paralist}
\usepackage[colorinlistoftodos]{todonotes}
\usepackage[all]{xy}
\usepackage{enumerate}
\usepackage{multirow}
\usepackage{wasysym}
\usepackage{latexsym} 
\usepackage{comment}
\usepackage{colonequals}

\usepackage{geometry}
\usepackage[
        draft=false,
        colorlinks, citecolor=darkgreen,
        backref,       
        pdfauthor={F.F. Favale, J.C. Naranjo, G.P. Pirola, S. Torelli},
        pdftitle={Coverings},
        linktocpage        
]{hyperref}
\hypersetup{citecolor=blue,linktocpage}

\newtheorem{theorem}{Theorem}[section]
\newtheorem*{theorem*}{Theorem}
\newtheorem*{definition*}{Definition}
\newtheorem{lemma}[theorem]{Lemma}
\newtheorem{corollary}[theorem]{Corollary}
\newtheorem{proposition}[theorem]{Proposition}
\newtheorem{remark}[theorem]{Remark}
\newtheorem{definition}[theorem]{Definition}
 
\newtheorem{conjecture}[theorem]{Conjecture}

\newcommand{\nc}{\newcommand}

\nc{\cH}{{\mathcal H}}
\nc{\cA}{{\mathcal A}}
\nc{\cG}{{\mathcal G}}
\nc{\cC}{{\mathcal C}}
\nc{\cD}{{\mathcal D}}
\nc{\cO}{{\mathcal O}}
\nc{\cI}{{\mathcal I}}
\nc{\cB}{{\mathcal B}}
\nc{\cY}{{\mathcal Y}}
\nc{\cK}{{\mathcal K}} 
\nc{\cX}{{\mathcal X}}
\nc{\cS}{{\mathcal S}}
\nc{\cE}{{\mathcal E}}
\nc{\cF}{{\mathcal F}}
\nc{\cZ}{{\mathcal Z}}
\nc{\cQ}{{\mathcal Q}}
\nc{\cN}{{\mathcal N}}
\nc{\cP}{{\mathcal P}}
\nc{\cL}{{\mathcal L}}
\nc{\cM}{{\mathcal M}}
\nc{\cT}{{\mathcal T}}
\nc{\cW}{{\mathcal W}}
\nc{\cU}{{\mathcal U}}
\nc{\cJ}{{\mathcal J}}
\nc{\cV}{{\mathcal V}}
\nc{\cR}{{\mathcal R}}
\nc{\bH}{{\mathbb H}}
\nc{\bA}{{\mathbb A}}
\nc{\bG}{{\mathbb G}}
\nc{\bC}{{\mathbb C}}
\nc{\bO}{{\mathbb O}}
\nc{\bI}{{\mathbb I}}
\nc{\bB}{{\mathbb B}}
\nc{\bY}{{\mathbb Y}}
\nc{\bK}{{\mathbb K}} 
\nc{\bX}{{\mathbb X}}
\nc{\bS}{{\mathbb S}}
\nc{\bE}{{\mathbb E}}
\nc{\bF}{{\mathbb F}}
\nc{\bZ}{{\mathbb Z}}
\nc{\bQ}{{\mathbb Q}}
\nc{\bN}{{\mathbb N}}
\nc{\bP}{{\mathbb P}}
\nc{\bL}{{\mathbb L}}
\nc{\bM}{{\mathbb M}}
\nc{\bT}{{\mathbb T}}
\nc{\bW}{{\mathbb W}}
\nc{\bU}{{\mathbb U}}
\nc{\bD}{{\mathbb D}}
\nc{\bJ}{{\mathbb J}}
\nc{\bV}{{\mathbb V}}
\nc{\bR}{{\mathbb R}}

\nc{\OO}{\mathcal{O}}
\nc{\PP}{\mathbb{P}}

\DeclareMathOperator{\Hom}{Hom}

\DeclareMathOperator{\Pic}{Pic}

\DeclareMathOperator{\Sing}{Sing}

\DeclareMathOperator{\Sp}{Sp}
\DeclareMathOperator{\Ab}{Ab}

\DeclareMathOperator{\MDM}{\overline{\cM}_g^{\mbox{{\fontsize{5}{12}\selectfont DM}}}}

\DeclareMathOperator{\ASat}{\overline{\cA}_g^{\mbox{{\fontsize{5}{12}\selectfont Sat}}}}

\DeclareMathOperator{\Mgo}{\cM_g^{o}}
\newcommand{\Sat}{\overline{\tau}}

\DeclareMathOperator{\Rgbar}{\overline{\cR}_g}

\DeclareMathOperator{\Sgbar}{\overline{\cS}_g}
\DeclareMathOperator{\Sgpbar}{\overline{\cS}_g^+}

\DeclareMathOperator{\debar}{\overline{\partial}}

\nc{\fA}{{\mathfrak{A}}}
\nc{\fB}{{\mathfrak{B}}}
\nc{\fC}{{\mathfrak{C}}}
\nc{\fD}{{\mathfrak{D}}}
\nc{\fE}{{\mathfrak{E}}}
\nc{\fF}{{\mathfrak{F}}}

\nc{\p}{\partial}
\nc{\ph}{\hat{\partial}}

\nc{\war}{{\color{red} CHECK}}

\author[F.F. Favale]{Filippo Francesco Favale}
\address[F.F. Favale]{Dipartimento di Matematica e Applicazioni,
	Universit\`a degli Studi di Milano-Bicocca,
	Via Roberto Cozzi, 55,
	I-20125 Milano, Italy}
\email{filippo.favale@unimib.it}

\author[J. C. Naranjo]{Juan Carlos Naranjo$^{1,2}$}
\address[J. C. Naranjo]{
1.
Departament de Matem\`atiques i Inform\`atica, 
    Universitat de Barcelona, Gran Via 585, 
    08007 Barcelona, Spain \newline 
2. Centre de Recerca Matem\`atica, Edifici C, Campus Bellaterra, 08193 Bellaterra, Spain
    }
\email{jcnaranjo@ub.edu}

\author[G.P. Pirola]{Gian Pietro Pirola}
\address[G.P. Pirola]{Dipartimento di Matematica,
	Universit\`a degli Studi di Pavia,
	Via Ferrata, 5
	I-27100 Pavia, Italy}
\email{gianpietro.pirola@unipv.it}

\author[S. Torelli]{Sara Torelli}
\address[S. Torelli]{
1. Institut f\"ur Algebraische Geometrie, 
    Leibniz Universit\"at Hannover, 
    Welfengarten 1, 30167 Hannover, Germany}
\email{torelli@math.uni-hannover.de}

\title{Holomorphic $1$-forms on some coverings of the moduli space of curves}

\dedicatory{Dedicated to Enrico Arbarello}

\begin{document}

\begin{abstract}
In this paper we consider unramified coverings of the moduli space $\mathcal M_g$ of smooth projective complex curves of genus $g$. Under some hypothesis on the branch locus of the finite extended map to the Deligne-Mumford compactification, we prove the vanishing of the vector space of holomorphic $1$-forms on the preimage of the smooth locus of $\mathcal M_g$. This applies to several moduli spaces, as the moduli space of curves with $2$-level structures, of spin curves and of Prym curves. In particular, we obtain that there are no non-trivial holomorphic $1$-forms on the smooth open set of the Prym locus.   

\end{abstract}

\thanks{
F.F Favale and G.P. Pirola are members of GNSAGA (INdAM) and are partially supported by PRIN project {\it Moduli spaces and Lie theory 2017} and by {\it MIUR: Dipartimenti di Eccellenza Program (2018-2022) - Dept. of Math. Univ. of Pavia}. J.C. Naranjo is partially supported by the Spanish MINECO research project {\it PID2019-104047GB-I00}. S. Torelli is supported by Alexander von Humboldt Foundation.}

\maketitle
\section*{Introduction}

In its beautiful paper \cite{Mum} Mumford shows that the first cohomology group of $\cM_g$, the moduli space of smooth projective complex curves, is torsion. This was perhaps the first important result on the topology of $\cM_g$ after the theorems of Hurwitz, Luroth, and Clebsh, which proved the connectedness of the moduli space. Mumford's result has been made precise and greatly generalized by Harer \cite{Har}, who also computed the second cohomology group of $\cM_g.$ An important consequence of Mumford computation is that the Albanese variety of the Deligne-Mumford compactification $\MDM$ is trivial and that there are no non-trivial {\em closed} 1-holomorphic form on it. In \cite{FPT} the authors were able to remove the closed condition. 
To be more precise it was shown that for $g\geq 5 $ (and a simple inspection gives that $g\geq 4$ is sufficient), $H^0(\Mgo,\Omega^1_{\Mgo})=0$ where we denote by $\Mgo$ the smooth locus of $\cM_g$, and by $\Omega^1_{\cM_g}$ its cotangent bundle. The proof was given by reducing the problem to some test surfaces via Satake compactification, together with a completion argument which set up the possible essential singularities of the forms.
\vskip 0,4cm

The work of  \cite{FPT}  was motivated by the theory of projective structures on curves. In fact, the vanishing of $H^0(\Mgo,\Omega^1_{\Mgo})$ provides a one-to-one correspondence
between $(1,1)$ $\debar$-closed forms with fixed cohomology class and the $\cC^{\infty}$ sections of the affine space of the projective structures on $\Mgo$ (see \cite{BCFP}). Moreover, this result allowed some simplification (see \cite{BFPT}) on the well-known fact that
the Weil-Petersson form on $\cM_g$ is associated to the projective structure obtained by uniformization Theorem.
\vskip 0,4cm

Besides, we believe that this type of conditions are really interesting on their own. They deal with  important
algebro-analytic-geometric properties of quasi-projectve varieties. From Hodge theory, one knows that  all the holomorphic $(p,0)$-forms  on a complete smooth complex algebraic variety are closed. In contrast, on the affine
varieties there are plenty of holomorphic non-closed forms. 
This leads us to give the following definition.

\begin{definition*} 
We say that an algebraic variety is  $p$-complete if all the holomorphic $(p,0)$-forms are closed: $H^0(X,\Omega^p_X)=H^0(X,\Omega^p_X)_{closed},$  i.e.  there are no non-zero $(p+1,0)$ holomorphic exact  forms on $X$.
\end{definition*}

With this language we have that for $g\geq 5$, $\Mgo$ is $1$-complete.
It is a {\em concavity result} in the sense of Andreotti and Grauert \cite{AG}  and one can conjecture that $\Mgo$ is $p$-complete or else that $H^0(\Mgo,\Omega_{\Mgo}^p)=0$ for  $g>>0.$ 
We believe that these problems should deserve to be investigated more thoroughly.
\vskip 0,4cm

In this paper we consider the problem of $(1,0)$-forms on \'etale finite covering $\pi:M\to \cM_g$ for $g\geq 5$. 
More precisely setting $M^o=\pi^{-1}(\Mgo)$, we ask whenever it holds that $H^0(M^o,\Omega^1_{M^o})=0.$

As in the case of $\cM_g$ we divide the problem in two by asking when:

\begin{enumerate}
 \item  $H^0(M^o,\Omega^1_{M^o})_{closed}=0.$
\item  $M^o$ is $1$-complete.
 \end{enumerate}

We note that $M^o$ is  $0$-complete  since $M$ is covered by complete curves; then the vanishing of the first statement above follows from the vanishing  of $ H^1(M^o,\bC).$
A famous conjecture of Ivanov (see \cite{Iv 15} or \cite{Kir}) predicts $H^1(M^o,\bC)=0$ for all finite \'etale covering of $\cM_g,$ and
the problem has been studied by many authors (see for example, \cite{Bog}, \cite{McC}, \cite{Hai}, \cite{HL}, \cite{Put}\dots). 
We find, in particular, very useful  the Boggi-Putman result (\cite{Bog}, \cite{Put}) which implies that $H^1(M^o,\bC)=0$ if the covering $\pi:M^o\to \Mgo$ is trivial around the Dehn twists associated to separating curves. To be more precise, let 
$\pi_\ast:\pi_1(M^o) \to \pi_1(\Mgo)$ be the map between the fundamental groups. One remarks that for $g>3$,
$\pi_1(\Mgo)$ is isomorphic to the mapping class group $\Gamma_g$. This follows because $\dim (\cM_g\setminus \Mgo)\leq 2.$ Let $J_g\subset \Gamma_g$ be Johnson group, i.e. the subgroup generated by the Dehn twists associated to separating curves.
As a consequence of the result of Boggi and Putman we have that $H^1(M^o,\bC)=0$ if $\pi_\ast(\pi_1(M^o))\supseteq J_g$. 
\medskip

To get a geometric condition, let $\overline{\pi}:\overline{M}\to \MDM$ be a finite covering, set $M^o=\overline{\pi}^{-1}(\Mgo)$ and assume that $\pi^o=\overline{\pi}|_{M^o}$ is \'etale. If $\overline{\pi}$ is ramified only on $\Delta_0$ we obtain
$(\pi^o)_\ast(\pi_1(M^o))\supseteq J_g$ and therefore we have the following:

\begin{theorem}[Boggi-Putman]
If $\overline{\pi}: \overline{M}\to  \MDM$ ramified only on $\Delta_0$ and $g>3$, then $H^1(M^o,\bC)=0.$\end{theorem}

Granted this topological result, next we need a condition that provides  the $1$-completeness of $M^o$. We show that this holds when $\overline{\pi}$ is not ramified on $\Delta_1$, the divisor where the general point represents a nodal curve of arithmetic genus $g$ which is the union of a curve of genus $g-1$ and an elliptic curve.
We show the following:

\begin{theorem}
If $\overline{\pi}: \overline{M}\to  \MDM$ does not ramify on $\Delta_1$ and $g\geq 5$, then $M^o$ is $1$-complete.
\end{theorem}

In order to prove this result, we follow an idea analogous to the one used in \cite{FPT} where we considered, as mentioned above, test surfaces which cover a dense open subset of $\Mgo$. Here, generalizing this to coverings of such surfaces, we obtain
a key technical tool (Lemma \ref{LEM:Lemma2b}).
We recall that these surfaces were obtained by intersecting a suitable embedding of the Satake compactification of $\cM_g$ in some projective space with general linear sections.
One notices that such surfaces intersect the boundary of $\MDM$ only on $\Delta_1$, and the condition of the previous theorem boils from this.

\vskip 0,4cm
Combining these results we  get the following.
\begin{theorem} 
Let $\overline{\pi}: \overline{M}\to  \MDM$ be covering that ramifies only on $\Delta_0$.
If $g\geq 5$ then $H^0(M^o,\Omega^1_{M^o})=0$. 
\end{theorem}

We apply the above results to the more usual coverings of $\cM_g$: moduli space with $2$-level structure, spin curves and moduli of the roots of canonical bundles.  
Finally, we deduce the $1$-completeness result for the Prym locus.
\vskip 0,4cm

The paper is organized in the following way. In section 1 we prove our $1$-completeness result on surfaces. In section 2 we recall the topological results of the vanishing of the first cohomology group. In section 3 we give our application to special coverings of $\cM_g$ (see Theorems \ref{thm:level}, \ref{thm:prym}, \ref{thm:prym2} and \ref{sgplus}).
\vskip 0,4cm

We would like to thank Gabriele Mondello for his enlightening explanation on the covering of the moduli space.
It is a pleasure to dedicate this work to Enrico Arbarello, one of the most profound connoisseurs of the geometry of the moduli space of curves.


\section{Holomorphic 1-forms on surfaces}
\label{SEC:One}

In this section, we study $1$-holomorphic forms on coverings of smooth projective surfaces in relation to birational morphisms. 
\medskip

We first consider the case of holomorphic forms on smooth projective surfaces endowed with a birational morphism satisfying some additional assumptions. The following lemma is a generalization of an argument used in \cite[Lemma 3.7, Theorem 1.6]{FPT} in the particular case of moduli of curves.

\begin{lemma}
\label{LEM:lemma1} Let $\varphi: S\to \bP^N$ be a birational morphism from a smooth projective surface over $\mathbb{C}$. Let $D=\sum_{i=1}^{k}D_i$ be the divisor contracted by $\varphi$ where $\{D_i\}_{i=1,\dots, k}$ are smooth pairwise disjoint divisors. Let $U=S\setminus D$ be its complement and let $C$ be a generic element of $|\cO_S(1)|$. Assume that 
\begin{equation}
\label{EQ:CondLemma1}
\deg((K_S+2D)|_{D_i})<0\qquad \mbox{ for all }\qquad i=1,\dots,k.
\end{equation} 
Then
\begin{enumerate}[(a)]
    \item $H^0(C,\Omega^1_{S}|_C)=H^0(U,\Omega^1_{S}|_U)=H^0(S,\Omega^1_S)$. In particular, holomorphic $1$-forms on $U$ are closed;
    \item The holomorphic functions on $U$ are constant. In particular, $H^0(U,\Omega^1_U)$ injects in $H^1(U,\mathbb C)$. 
\end{enumerate}
In other words, we have proved that $U$ is $0$-complete and $1$-complete.
\end{lemma}

\begin{proof}
In order to prove $(a)$, we first show that $H^0(S,\Omega^1_{S})= H^0(S,\Omega^1_S(mD))$, for any $m\geq 1$, and then that this implies $H^0(S,\Omega^1_{S}|_C)=H^0(U,\Omega^1_{S}|_U)=H^0(S,\Omega^1_S)$.

The proof goes by induction on $m$. Let $m=1$. Consider the short exact sequence 
\begin{equation}\label{eq1}0\to \Omega^1_S\to \Omega^1_S(D)\to \Omega^1_S(D)_{|D}\to 0\end{equation} 
defined by the inclusion $D\subset S$ twisted by $\Omega_S^1$. It's enough to show that the coboundary morphism $\partial : H^0(D,\Omega^1_S(D)|_{D})\to H^1(S,\Omega^1_S)$ on the long exact sequence in cohomology is injective. To prove this, consider the sequence
\begin{equation}\label{eq2}0\to \mathcal{O}_D\to \Omega^1_{S}|_D(D)\to \omega_D(D)\to 0,\end{equation} 
defined as the cotangent sequence twisted by $D$. By assumption, $D$ is a disjoint union of smooth divisors with $\omega_{D_i}(D)=\mathcal{O}_S(K_S+2D)|_{D_i}$ of negative degree. Hence $H^0(D_i,\omega_{D_i}(D))=0$ and so $H^0(D,\omega_D(D))=0$. Thus, we conclude from the long exact sequence in cohomology that $H^0(D,\mathcal{O}_D)\xrightarrow{\alpha} H^0(D,\Omega^1_{S|D}(D))$ is an isomorphism. Consider the diagram: 
\begin{small}
$$\xymatrix@!R{
			  {0}    &              {H^0(S,\Omega_S^1) }      &    {  H^0(S,\Omega_S^1(\log D))} &    {H^0(D,\mathcal{O}_D)}      & {H^1(S,\Omega_S^1)}  &\\
			{0}  & {H^0(S,\Omega_S^1)}  & {H^0(S,\Omega_S^1( D))}  & {H^0(D,\Omega^1_S(D)|_{D})}  & {H^1(S,\Omega_S^1)} &   
			\ar"1,1";"1,2"\ar"1,2";"1,3"\ar"1,3";"1,4"\ar"1,4";"1,5"^{\partial'}
			\ar@{=}"1,2";"2,2"\ar@{^{(}->}"1,3";"2,3"\ar@{=}"1,4";"2,4"^{\alpha}\ar@{=}"1,5";"2,5"
			\ar"2,1";"2,2"\ar"2,2";"2,3"\ar"2,3";"2,4"\ar"2,4";"2,5"^-{\partial}
			\hole
}$$
\end{small}
induced in cohomology by relating the exact sequences of $\log$ differentials to the sequence introduced above using $\alpha$. The connecting morphism $\partial$ is injective if the connecting morphism $\partial'$ is. As $D=\sum D_i$, we have $H^0(\mathcal{O}_D)= \bigoplus_i H^0(\mathcal{O}_{D_i})$. The morphism $\partial'$ sends the unique generator of each $H^0(\mathcal{O}_{D_i})$ to the Chern class $c_1(D_i)$ of $D_i$. Since $D_i$ are pairwise disjoint and contracted by $\phi$, then $D_i^2<0$. Therefore, $c_1(D_i)$ are non-zero and independent:  given a relation $\sum n_ic_1(D_i)=0$ with $n_i\in \mathbb Z$ then take the product with $c_i(D_i)$  and we obtain $n_i=0$. This gives the
  injectivity of $\partial '$ and concludes the case $m=1$.

\medskip	

We now assume that the statement true for $m-1$ and we prove it for $m$. More precisely we show that $H^0(S,\Omega^1_{S}((m-1)D))=H^0(S,\Omega^1_{S}(mD))$ for any $m> 1$. Consider the short exact sequence 
$$0\to H^0(S,\Omega^1_S((m-1)D))\to H^0(S,\Omega^1_S(mD))\to H^0(D,\Omega^1_S(mD)_{|D})$$
induced in cohomology by Sequence \eqref{eq1} twisted by $\cO_S((m-1)D)$. Using Sequence \eqref{eq2} twisted by $\cO_S((m-1)D)$, since $\mathcal{O}_D((m-1)D)$ and $\omega_D(mD)$ have negative degrees on each component of $D$, we get $H^0(D,\Omega^1_{S}(mD)|_D)=0$ and thus the thesis.

\medskip
We now prove that $H^0(C,\Omega^1_{S}|_C)=H^0(U,\Omega^1_{S}|_U)=H^0(S,\Omega^1_S)$. By \cite[Theorem 1.2]{FPT}, having $H^0(S,\Omega^1_{S})= H^0(S,\Omega_S(mD))$ for any $m\geq 1$, implies $H^0(S,\Omega^1_S)\simeq H^0(C,\Omega^1_{S}|_C)$.  By \cite[Theorem 1.6]{FPT}, we also conclude that $H^0(U,\Omega^1_{S}|_U)=H^0(S,\Omega^1_S)$ (in the language of \cite{FPT}, $\Omega^1_S$ is equivalently $\mathcal{O}_S(C)$-liftable and $\mathcal{O}_S(C)$-concave). This ends the proof of $(a)$.

\medskip
To prove $(b)$, let $f\in H^0(U,\mathcal{O}_U)$ be a holomorphic function. Consider the images $\varphi(S)$ and $\varphi(U)$ of $S$ and $U$ respectively via $\varphi$. Since $\varphi(S)$ is covered by complete curves of genus at least $2$ and $\varphi(U)$ is nothing but $\varphi(S)$ minus a certain finite number of points, then an open dense subset of $\varphi(U)$ is covered by complete curves and so it is an open dense subset of $U$. But now the restriction of $f$ to any such curve must be constant, and so $f$ must be constant as well.
\end{proof}

We now focus on coverings of surfaces satisfying the assumptions of Lemma \ref{LEM:lemma1}.  

\begin{lemma}
\label{LEM:Lemma2b}
Let $S,D,U$ and $\varphi$ be as in Lemma \ref{LEM:lemma1} and consider a finite morphism  $\pi:\hat{S}\to S$ with ramification divisor $R$. Set $\hat{D}=\pi^{-1}(D)$ and $\hat{U}=\pi^{-1}(U)$ and let $\hat{C}$ be a general element in the linear system of a sufficiently high multiple of $\pi^*(\cO_{S}(1))$.
Assume that $R$ is disjoint with $\hat{D}$. Then,
\begin{enumerate}[(a)]
    \item $H^0(\tilde C,\Omega^1_{\hat{S}}|_{\hat{C}})=H^0(\tilde U,\Omega^1_{\hat{S}}|_{\hat{U}})=H^0(\tilde S,\Omega^1_{\hat{S}})$. In particular, holomorphic $1$-forms on $\hat{U}$ are closed;
    \item the holomorphic functions on $\hat{U}$ are constant. In particular,   $H^0(\hat{U},\Omega^1_{\hat U})$ injects in $H^1(\hat U, \mathbb C)$. 
\end{enumerate}
\end{lemma}

\begin{proof}
Since $\pi$ is finite and $\cO_{S}(1)$ is semiample, big and nef by assumption, the same holds for $\pi^*(\cO_{S}(1))$. Then, a sufficiently high multiple gives us a birational morphism $\hat{\varphi}$ contracting exactly $\hat{D}$.

\vspace{2mm}

We now prove that $\hat{\varphi}$ satisfies the assumptions of Lemma \ref{LEM:lemma1}. By Riemann-Hurwitz, we have that $K_{\hat{S}}=\pi^*K_{S}+R$ so $K_{\hat{S}}+2\hat{D}=\pi^*(K_S+2D)+R$. On the other hand, since $\hat{D}$ is the sum of smooth disjoint divisors $\hat{D}_i$ (as  $R$ is disjoint with $\hat{D}$ by assumption) we have that  $\deg((K_{\hat{S}}+2\hat{D})|_{\hat{D}_i})$ is given by a positive multiple (i.e. the degree of the \'etale covering $\pi|_{\hat{D}_i}:\hat{D}_i\to D_i$) of $\deg((K_S+2D)|_{D_j})$. Then, for all $i$, we have $\deg((K_{\hat{S}}+2\hat{D})|_{\hat{D}_i})<0$ so we can conclude by applying Lemma \ref{LEM:lemma1}.
\end{proof}


\section{Application to coverings of moduli of curves}
\label{SEC:CoveringOfModuli}

In this section we apply the results of the previous section to suitable coverings of the Deligne-Mumford compactification $\MDM$ of the moduli space $\cM_g$ of smooth curves of genus $g\geq 5$.

\subsection{Holomorphic $1$-forms on covering of the moduli space of curves}

We now recall some notions about these moduli spaces, referring to \cites{HM,ACG2}.

We denote by $\Mgo$ the smooth locus of $\cM_g$, which is the locus of $\cM_g$ of curves with trivial automorphisms group and differs from $\cM_g$ by loci of codimension at least $2$ since $g\geq 5$. Let $\MDM$ be the Deligne-Mumford compactification of $\cM_g$, i.e. the moduli space parametrizing classes of stable curves of arithmetic genus $g$.
The boundary of $\MDM\setminus \cM_g$ is given by the union $$\Delta_0\cup \Delta_1\cup \cdots \cup \Delta_{[g/2]}$$ of divisors of $\MDM$. The general element of $\Delta_0$ is given by the class of a irreducible curve with a single node; the general element of $\Delta_i$ is represented by a curve with two smooth irreducible components of genus $i$ and $g-i$ respectively, meeting transversally at a single point.
\medskip

We recall that there is a morphism $\varphi:\MDM\to \bP^N$, birational on its image, contracting the boundary divisors of $\MDM$. If we denote by $\ASat$ the Satake compactification of $\cA_g$ (see \cite{Sat}), the above map is given by the composition of the natural extension $\MDM\to \ASat$ of the Torelli morphism with the embedding induced by the Hodge class of $\ASat$. The fiber over a general point of $\Delta_i$ has dimension $2$ for $i\neq 1$ and $1$ for $i=1$. Indeed, assume $i>0$, then the image is in the closure of the product of the Jacobian loci $\mathcal J_i\times \mathcal J_{g-i}$ in $\mathcal A_i \times \mathcal A_{g-i}$, which have dimension $3g-6$ for $i>1$ and $3g-5$ for $i=1$. On the other hand, the image of $\Delta_0$ is in the closure of $\mathcal J_{g-1}$ as subvariety of the Satake compactification of $\mathcal A_g$, which has dimension $3g-6$.
\medskip

We introduce the following objects.

\begin{definition}
A complex surface in $\MDM$ is a {\it test surface} if it is the strict transform of a general complete intersection of hyperplanes in a suitable projective immersion of the Satake compactification. Analogously, one defines {\it test curves}.
\end{definition}

Test surfaces have been studied in \cite[Proposition 3.5]{FPT} (under the name $H$-surfaces) and played a central role in the proof of the following

\begin{theorem}\cite[Theorem 3.1]{FPT}
\label{THM:MAINFPT}
The space $H^0(\Mgo,\Omega^1_{\Mgo})$ vanishes for $g\geq 5$.
\end{theorem}

The importance of these surfaces for the proof of the above theorem relies on the fact that they are smooth compact surfaces intersecting the boundary divisors of $\MDM$ in $\Delta_1$. Instead, test curves are fully contained in $\cM_g$. 
\medskip

We prove now a generalization of the above for coverings.

\begin{theorem}
\label{THM:MAIN}
Let $\pi:M\to\MDM$ be a finite morphism with branch locus supported on $\Delta\setminus \Delta_1$. Assume $g\ge 5$. Let $M^o=\pi^{-1}(\Mgo)$. Then:
\begin{enumerate}[(a)]
\item all holomorphic functions on $M^o$ are constant (i.e. $\Mgo$ is $0$-complete);
\item holomorphic $1$-forms on $M^o$ are closed (i.e. $\Mgo$ is $1$-complete);
\item if $H^1(M^o,\mathbb C)=0$, then we have $H^0(M^o,\Omega^1_{M^o})=0$.
\end{enumerate}
\end{theorem}

\begin{proof}
Notice that $\pi|_{M^o}$ is \'etale and so, in particular, $M^o$ is smooth. Furthermore, recall that $\Mgo$ can be covered by compact curves (see \cite{HM}) which are nothing but our test curves as defined above (see also \cite{FPT}). Therefore, $M^o$ is covered by compact curves too, by considering coverings of test curves. In particular, from this it follows that both $\Mgo$ and $M^o$ have no non-constant holomorphic functions and (a) follows.
\medskip

Now we prove (b).  Let $\omega\in H^0(M^o,\Omega_M^1|_{M^o})$ and assume, by contradiction, that $d\omega\neq 0$. Then, for any $q\in M^o$, there exists a $2$-plane $\sigma$ in $T_{M,q}$ such that $d\omega(q)|_\sigma\neq 0$. Furthermore, we have an isomorphism $d_q\pi:T_{M,q}\to T_{\Mgo,p}$ where $\pi(q)=p$. Therefore, we can find a test surface $S$ passing through $p$ with tangent space $d_q\pi(\sigma)$. By \cite[Proposition 3.5]{FPT}, $S$ is a smooth regular surface that does not meet any divisor $\Delta_i$ for $i\neq 1$ and $E=\Delta_1\cap S$ is a disjoint union of smooth curves of genus $g-1$. Moreover, $\deg((K_S+2E)|_{E})<0$ so $\varphi|_S$ satisfies the hypothesis of Lemma \ref{LEM:lemma1}.
\vspace{2mm}

Since $\pi$ has branch supported on $\Delta\setminus \Delta_1$ we have that $\pi|_{\pi^{-1}(S)}:\pi^{-1}(S)\to S$ is \'etale. Let $\hat{S}$ be the connected component of $\pi^{-1}(S)$ containing $q$. Then, $\pi|_{\hat{S}}:\hat{S}\to S$ is \'etale, $\hat{S}$ is smooth and $T_{\hat{S},q}=\sigma$ by construction. 
\vspace{2mm}

Setting $\hat{E}=\pi|_{\hat{S}}^{-1}(E)$, we have that $(\hat{S}\setminus\hat{E})\subseteq \hat{U}$ so that it makes sense to restrict $d\omega$ to $\hat{S}\setminus\hat{E}$. By assumption $d\omega|_{\hat{S}\setminus\hat{E}}\neq 0$ since
$$d\omega|_{\hat{S}\setminus\hat{E}}(q)=d\omega(q)|_\sigma\neq 0.$$
But this is impossible since, by Lemma \ref{LEM:Lemma2b}, all holomorphic $1$-forms on $\hat{S}\setminus \hat{E}$ are closed. Hence, $\omega$ is closed, as claimed.
\vspace{2mm}



Finally, we prove (c). Let $\omega\in H^0(M^o,\Omega^1_{M^o})$. By $(b)$, $\omega$ is a closed $1$-forms so $[\omega]\in H^1(M^o,\bC)$. Since, by assumption, $H^1(M^o,\bC)=0$ we conclude that $\omega$ is exact. Then, by $(a)$, $\omega = 0 $ as claimed.
\end{proof}

\begin{remark}
Notice that  without the hypothesis in (c) on the cohomology, we only can state that  the possible non-zero $1$-forms on $M^{o}$ are not pull-back of forms on $\Mgo$.
\end{remark}


\subsection{About the first cohomology group of covering of $\Mgo$}

The computation of the first cohomology group of a finite cover of $\cM_g$ is an interesting non-trivial problem that has been intensively studied in literature. For instance, it is a deep result due to Harer (see \cite{Har}) that the cohomology of the moduli of Spin curves is zero. Nonetheless, for other important moduli spaces of curves with additional structures the first cohomology group has not been directly computed. 
\medskip

There are several important results in the direction of computing the first cohomology group of finite coverings. For instance, we refer to the works of Boggi, MacCarthy, Hain, Loojenga, and Putman (see \cite{Bog},\cite{McC},\cite{Hai},\cite{HL},\cite{Put}). In particular, we will use two theorems (Theorems \ref{McC} and \ref{THM:PUTBOG}). Let us first recall some basic facts on the fundamental group of $\cM_g$.


Let $S$ be a compact connected oriented topological surface of genus $g$ without boundary and consider the associated Teichmüller space $\mathcal{T}_g$. Let $\Sigma_g$ be the group of orientation preserving homeomorphisms of $S$ modulo isotopy, i.e. the mapping class group of $S$. We have an action of $\Sigma_g$ on $H^1(S,\bZ)$ which preserves the cup product pairing; this induces a surjective homomorphism of groups 
$$\delta:\Sigma_g\to \Sp_{2g}(\bZ)$$
whose kernel is the {\it Torelli group} $\mathcal J_g$, i.e. the group of homomorphism classes which act trivially on the homology of $S$.
It is well known that $\Sigma_g$ acts on $\mathcal{T}_g$ in a properly discontinuous way and that the quotient can be identified with $\cM_g$.
It is known that, for $g\ge 5$, $\Sigma_g$ is the fundamental group of $\Mgo$. Then,
the Torelli group $\mathcal J_g$ can be seen as a subgroup of $\pi_1(\Mgo)\simeq \Sigma_g$.

The following result is the main theorem in \cite[Theorem 2.1]{McC} and of \cite{HL}:

\begin{theorem}\label{McC}
Let $G$ be a subgroup of $\Sigma_g$ of finite index containing $\mathcal J_g$. Then $H^1(G,\mathbb Z)=0$.
\end{theorem}
In the previous theorem, $H^1(G,\mathbb Z)$ stands for the cohomology in group theory with integer coefficients. By standard results, this group is isomorphic to $\Hom_{\mathbb Z}(\Ab(G),\mathbb Z)$.
\medskip 

A stronger result involving the group cohomology with rational coefficients is due to Boggi and Putnam (see \cite{Put} and \cite{Bog}). In order to introduce this result, let us recall some further information about the generators of the Torelli group $\cJ_g$.
\medskip

It is well known that the mapping class group of a compact connected oriented surface $S$ is generated by Dehn twists. If $\gamma$ is a simple loop on $S$, we will denote by $D_\gamma$ the associated Dehn twist. One says that a simple loop $\gamma$ is a {\it non-separating loop} if, by removing its support, $S$ remains connected. A pair of loops $(\gamma,\delta)$ is said to be a {\it bounding pair} if $\gamma$ and $\delta$ are homologous, non-homotopic non-separating simple loops which, together, disconnect the surface (to be more precise, one also needs to consider a suitable orientation). 
One defines the {\it Johnson group} $\cK_g$ to be the subgroup of $\Sigma_g$ generated by Dehn twists associated to separating loops. Since in \cite{Bir} and \cite{Pow} it is shown that $\cJ_g$ is generated by
$$\{D_\gamma\,|\, \gamma \mbox{ is a separating loop}\}\qquad  \mbox{ and }\qquad \{D_\gamma D_{\delta}^{-1}\,|\, (\gamma,\delta) \mbox{ is a bounding pair}\}.$$
Then, it is clear that the Johnson group is a subgroup of the Torelli group.
\medskip

The fundamental result we mentioned before is

\begin{theorem}[\cite{Put},\cite{Bog}]
\label{THM:PUTBOG}
Let $G$ be a finite index subgroup of $\Sigma_g$ containing $\cK_g$. Then $H^1(G,\bR)=0$.
\end{theorem}

It is worth to recall that these two theorems are evidence for the following conjecture (see \cite[Section 7]{Iv15}):
\begin{conjecture}[Ivanov]
All finite index subgroups $G$ of $\Sigma_g$ satisfy $H^1(G,\bZ)=0$.
\end{conjecture}

Now we prove:

\begin{proposition}
\label{PROP:ETALETORELLI}
Let $\pi:M\to\MDM$ be a finite morphism with branch divisor $B$ and set $M^o=\pi^{-1}(\Mgo)$ and $\pi^o=\pi|_{M^o}$. Let $G={\pi^o}_*(\pi_1(M^o))$ be the image of the fundamental group of $M^o$. Assume that $g\geq 5$, $\cK_g\subseteq G$ and that the branch $B$ has support on $\Delta\setminus \Delta_1$. Then $H^0(M^o,\Omega^1_{M^o})=0$.
\end{proposition}

\begin{proof}
Since $\pi^o$ is \'etale, and it is associated to the subgroup
$G$ of $\pi_1(\Mgo)\simeq \Sigma_g$, we have that  $\pi_1(M^o)\simeq G$. Then, by the Universal coefficients Theorem yield 
\begin{equation}\label{univ_coefficients}
H^1(M^o,\bZ)\simeq \Hom_{\bZ}(H_1(M^o,\bZ),\bZ)\simeq \Hom_{\bZ}(\Ab(G),\bZ)\simeq H^1(G,\bZ).
\end{equation}
\medskip
By assumption, $G$ is a subgroup of $\Sigma_g$ of finite index that contains $\cK_g$ so, by Theorem \ref{THM:PUTBOG}, we have $H^1(M^o,\bC)=0$. We can conclude using Theorem \ref{THM:MAIN}$(c)$.
\end{proof}

\begin{lemma}
\label{LEM:JOHNSON}
Let $\pi:M\to\MDM$ be a finite morphism with branch divisor $B$ whose support is contained in $\Delta_0$. Assume that $g\ge 5$. Set $M^o=\pi^{-1}(\Mgo)$ and $\pi^o=\pi|_{M^o}$. Let $G={\pi^o}_*(\pi_1(M^o))$ be the image of the fundamental group of $M^o$. Then 
$$B \mbox{ has support in } \Delta_0 \Longrightarrow \cK_g\subseteq G.$$
\end{lemma}

\begin{proof}
Assume that $B$ has support in $\Delta_0$. Let $\alpha\in \pi_1(\Mgo)$ be an element corresponding to a Dehn twist associated to a separating loop. Then $\alpha$ can be represented as a simple loop around a divisor $\Delta_i$ with $i\neq 0$. By assumption $\Delta_i$ is not contained in the branch locus so $\pi^o$ is a homeomorphism near the preimage of $\Delta_i$. In particular, $\alpha$ is lift to a simple closed arc in $M^o$. Since elements of $G$ are exactly loops of $\Mgo$ which lift to loops, we have that $\alpha\in G$. As $\cK_g$ is generated by separating loops, we have $\cK_g\subseteq G$.
\end{proof}

An immediate consequence of Lemma \ref{LEM:JOHNSON} is the following:

\begin{corollary}
\label{COR:BRANCHONDELTA0}
Let $\pi,\pi^o, M^o$ and $G$ be like in Theorem \ref{PROP:ETALETORELLI}. If we assume that $g\geq 5$ and that $B$ has support on $\Delta_0$, we have $H^0(M^o,\Omega^1_{M^o})=0$.
\end{corollary}

\begin{proof}
It is enough to observe that by Lemma \ref{LEM:JOHNSON} we have that $\cK_g\subseteq G$. Then one conclude again using Proposition \ref{PROP:ETALETORELLI}.
\end{proof}

\begin{remark}
Assuming the Ivanov's conjecture, for any $\pi$ as in Theorem \ref{THM:MAIN} ramified at most on $\Delta \setminus \Delta_1$, we would have $H^1(M^o,\bZ)=H^0(M^o,\Omega_{M^o}^1)=0$ by Proposition \ref{PROP:ETALETORELLI}.
\end{remark}


\section{Holomorphic $1$-forms of moduli spaces of curves with level structure}\label{SEC:LEVEL}

In this section we apply the results of Section \ref{SEC:CoveringOfModuli} to the moduli space $\cM_g[2]$ curves of genus $g$ with $2$-level structure. We work under the assumption that $g\geq 5$.
\medskip

We first shortly summarize some facts on these moduli spaces, we refer to \cite{ACG2}. 

Let $\cM_g[2]$ be the moduli space of curves of genus $g$ with $2$-level structure, the space parametrizing isomorphism classes $[(C,\rho)]$ where $C$ is a smooth curve of genus $g$ and $\rho:H_1(C,\bZ_2)\to \bZ_2^{\oplus 2g}$ is a symplectic isomorphism with a fixed symplectic bilinear form on $\bZ_2^{\oplus 2g}$ and the pairing induced on $H_1(C,\bZ_2)$ by the intersection pairing on $C$. The space $\cM_g[2]$ is normal
and admits a natural forgetful morphism $\pi':\cM_g[2]\to \cM_g$ which is a finite cover of $\cM_g$ unramified on $\Mgo$.

Consider the compactification $\overline{\cM}_g[2]$ of $\cM_g[2]$ as the normalization of $\MDM$ in the field of rational functions of $\cM_g[2]$ which gives the following diagram

$$
\xymatrix{
\cM_g[2] \ar[d]_{\pi'}\ar@{^{(}->}[r] & \overline{\cM}_g[2]\ar[d]^{\pi}\\
\cM_g \ar@{^{(}->}[r] & \MDM
}
$$

One can show that the morphism $\pi$ is unramified in codimension $1$ outside the inverse image of $\Delta_0$. Notice that this fact is no longer true for $m$-level structures with $m$ greater than two due to the automorphism of order $2$ on the generic point of $\Delta_1$. 
More in detail, to see that the morphism is unramified in codimension $1$ outside $\Delta_0$, notice that the Torelli map $\tau: \cM_g\to \cA_g$ extends to $\Sat:\MDM\to \ASat$ and, for $i\neq 0$, the generic point of $\Delta_i$ has image in $\cA_g$. 

The general point $[C_i\cup_p C_{g-i}]$ of $\Delta_i$, for $i\neq 0$, has in fact, as image, the abelian variety $J(C_i)\times J(C_{g-i})$, which lies in $\cA_g$. 

Recall that a $2$-level structure on an abelian variety $A$ is given by the choice of a symplectic isomorphism $$\rho:H_1(A,\bZ_2)\to \bZ_2^{\oplus 2g}.$$
In particular, a level structure on a Jacobian  corresponds to a level structure on the associated curve, since the homology of a Jacobian is exactly the homology of the associated curve. In other words, if $\cA_g[2]$ denotes the moduli space of abelian varieties with $2$-level structures, we have a commutative diagram

$$
\xymatrix{
\cM_g[2] \ar[d]_{\pi'}\ar[r] & \cA_g[2]\ar[d]\\
\cM_g \ar[r]_{\tau} & \cA_g
}
$$
which extends on $\MDM \setminus\Delta_0$. Consequently, for $m=2$ where the automorphism of order $2$ on the generic point of the divisor $\Delta_1$ does not produce ramification, we conclude that the branch divisor intersects at most on $\Delta_0$.

Since $\pi^o$ is \'etale we have that $G$ is a finite index subgroup of $\pi_1(\Mgo)=\Sigma_g$. Hence, $\pi:\overline{\cM}_g[2]\to \MDM$ is a finite morphism with branch supported at most on $\Delta_0$ so we are in the setting of Corollary \ref{COR:BRANCHONDELTA0}. Then we obtain the following:

\begin{theorem}\label{thm:level}
For $g\geq 5$ and $m\geq 2$, $\Mgo[2]=\bar{\pi}^{-1}(\Mgo)\subset \overline{\cM}_g[2]$ has no holomorphic 1-forms, that is  $H^0(\Mgo[2],\Omega^1_{\Mgo[2]})=0$.
\end{theorem}

\section{Holomorphic one forms on moduli of Prym curves and Prym locus}
\label{SEC:PRYM}

In this section we apply the results of Sections \ref{SEC:CoveringOfModuli} and \ref{SEC:LEVEL} to the moduli space $\cR_g$ of unramified double covers of curves of genus $g$ and to the Prym locus  of Prym varieties in the moduli space  of principally polarized abelian varieties.
\medskip

Let $\cR_g$ be the moduli space of unramified double covers of curves of genus $g$, the space parametrizing classes of pairs $[(C,\eta)]$ modulo automorphisms where $C$ is a smooth curve of genus $g$ and $\eta\in \Pic^{0}(C)\setminus\{\cO_C\}$ such that $\eta^2=\cO_C$. The natural forgetful morphism $\pi:\cR_g\to \cM_g$ is an unramified covering of degree $2^{2g}-1$.

Let $\cM_g[2]$ be the moduli space of curves of genus $g$ with level-$2$ structure. There is a natural relation between $\cR_g$ and $\cM_g[2]$. Indeed, a $2$-level structure on a smooth curve $C$ is a symplectic isomorphism $H_1(C,\bZ_2)\simeq \bZ_2^{2g}$, which by definition fixes the $2$-torsion of the homology. There is a natural bijection between $\Pic^0(C)[2]$ and $H_1(C,\bZ_2)$, and so elements of $\cR_g$ fixes one non-trivial element of $H_1(C,\bZ_2)$. This allows to compare the finite morphisms
$$\pi_1:\cR_g\to \cM_g\quad \mbox{ and }\quad \pi_2:\cM_g[2]\to \cM_g.$$
More precisely, on the restriction to the loci $\cR_g^o=\pi_1^{-1}(\Mgo)$, $\cM_g^o[2]$ (where the restriction $\tilde{\pi}_i$ of $\pi_i$ is unramified) we have the following diagram
\begin{equation}
\label{DIAG:PRYM}
\xymatrix{
\cM_g^o[2] \ar[dd]^{\tilde{\pi}_2}_{/\Sp_{2g}(\bZ_2)}\ar@{->>}[rd]^{\pi_{12}}\\
 & \cR_g^o\ar[ld]^{\tilde{\pi}_1}\\
\Mgo
}
\end{equation}
where the fiber of $\pi_{12}$ over $[(C,\eta)]$ are all the elements $[(C,\delta)]$ where $\delta(\eta)=\eta$.

This consideration leads to the following result as an application of Theorem \ref{thm:level}.

\begin{theorem}\label{thm:prym}
For $g\geq 5$,  $\cR_g^o=\pi_i^{-1}(\Mgo)\subset R_g$ has no holomorphic 1-forms, i.e.  $H^0(\cR_g^o,\Omega^1_{\hat{\cR_g^o}})=0$.
\end{theorem}

\begin{proof}
It is enough to observe that the morphism $\pi_{12}$ in Diagram \ref{DIAG:PRYM} is dominant and so, in particular, non-zero holomorphic one forms on $\cR_g^o$ would induce holomorphic one forms on $\cM_g^o[2]$ contradicting Theorem \ref{thm:level} proved in Section \ref{SEC:LEVEL}. 
\end{proof}

An alternative strategy to prove the theorem would use Theorem  \ref{THM:MAIN} on the finite map $\pi:\Rgbar\to \MDM$, where $\Rgbar$ is the compactification of $\cR_g$ introduced in \cite{Bea,BCF}.
This space parametrizes Prym curves (see \cite{BCF}), i.e. classes of triples $(X,\eta,\beta)$ where $X$ is a quasi-stable curve of genus $g$, $\eta$ is a line bundle on $X$ of total degree $0$ with $\eta|_E=\cO_E(1)$ for each exceptional component $E$ of $X$, and $\beta$ is a sheaf homomorphism $\eta^2\to \cO_X$ which is generically non-zero along all non-exceptional irreducible components of $X$.

It is proven that the forgetful morphism $\cR_g\to \cM_g$ extends to a ramified morphism $\pi: \Rgbar\to \MDM$ which associates to $[(X,\eta,\beta)]$ the class $[st(X)]$ of the stable model of $X$. In \cite{FL} it is shown that the branch locus coincides with the divisor $\Delta_0\subset \MDM$, and that the ramification is described by the formula $$K_{\Rgbar}=\bar{\pi}^*(K_{\MDM})+\delta^{ram}_0.$$
Here $\delta_0^{ram}$ is the class of the divisor $\Delta_0^{ram}\in \Pic(\Rgbar)$, whose general point $[(X,\eta,\beta)]$ satisfies the following:
\begin{itemize}
\item $X=C\cup E$ where is $C$ a general smooth curve of genus $g-1$ and $E$ is a $\bP^1$ meeting $C$ transversely in two distinct points $p$ and $q$;
\item $\eta|_{C}^{\otimes2}=\cO_C(p+q)$.
\end{itemize}
This proves that we are under the assumptions of Theorem \ref{THM:MAIN}, namely, that the branch locus is contained in $\Delta\setminus\Delta_0$. 

We now focus on the Prym locus $\cP_{g}\subset \mathcal A_{g}$, which is  the image of the Prym map $Pr_{g+1}:\cR_{g+1}\to \cP_{g}$. 
It is known that the differential of the Prym map is injective out of the locus of pairs $[(C,\eta)]$ such that the Clifford index of $C$ is $\le 2$ (see \cite{Nar}). Since the tetragonal locus has dimension $2g+3$ and, for $g=10$, the locus of plane sextics has dimension $19$, we obtain that the codimension of the locus where the differential of the Prym map is not injective has codimension at least $3$ for $g\ge 9$. Therefore, assuming $g\ge 8$, we can mimic the argument with test surfaces in Sections \ref{SEC:One} and \ref{SEC:CoveringOfModuli} by assuming that these surfaces are contained in the open set $\mathcal U \subset \mathcal R_{g+1}$ where the differential is injective. We obtain that there are not differential forms in the intersection  $\mathcal R_{g+1}^o \cap \mathcal U$ and that the image of this intersection is contained in the smooth locus,  $\mathcal P_g^o$, of $\mathcal P_g$. This implies the following result.

\begin{theorem}\label{thm:prym2}
For $g\geq 8$, $\cP_g^o$ has no holomorphic 1-forms, i.e.  $H^0(\cP_g^o,\Omega^1_{\cP_g^o})=0$.
\end{theorem}


\section{Holomorphic $1$-forms of moduli space of even spin curves}
\label{SEC:SPIN}

In this section we apply the results of Section \ref{SEC:CoveringOfModuli} to the moduli space $S_g^+$ of even spin curves of genus $g$.
\medskip

We recall that a quasi-stable curves (see \cite{Cor}) is just a nodal semistable curve $X$ such that any two exceptional components never meets. The stable model of a quasi-stable curve $X$ is obtained by contracting all its exceptional components and is denoted by $st(X)$.
\medskip

Let $\cS_g$ be the moduli space of spin curves, the space parametrizing classes of pairs $[(C,\theta)]$ modulo automorphisms where $C$ is a smooth curve of genus $g$ and $\theta$ is a theta-characteristic, i.e. $\theta\in \Pic^{g-1}(C)$ such that $\theta^2=\omega_C$. Furthermore, $\cS_g$ has two disjoint connected components $\cS_g^+$ and $\cS_g^-$ parametrizing even and odd theta-characteristics respectively. The natural forgetful morphism $\pi:\cS_g\to \cM_g$ is an unramified covering of degree $2^{2g}$ that restricts to $\pi_+:=\pi|_{\cS_g^+}$ of degree $2^{g-1}(2^{g}+1)$.

There is a compactification $\Sgbar$ of $S_g$ introduced by \cite{Cor} parametrizing classes of triples $(X,\theta,\beta)$ where $X$ is a quasi-stable curve of genus $g$, $\theta$ is a line bundle on $X$ of total degree $g-1$ with $\theta|_E=\cO_E(1)$ for each exceptional component $E$ of $X$, and $\beta$ is a sheaf homomorphism $\theta^2\to \omega_X$ which is generically non-zero along all non-exceptional irreducible components of $X$.

In \cite{Cor} it is proven that the morphism $\pi$ extends to a ramified morphism $\bar{\pi}: \Sgbar\to \MDM$ which associates to $[(X,\theta,\beta)]$ the class $[st(X)]$ of the stable model of $X$. We concentrate now in the even part $\Sgbar ^+$ and in the restriction map $\bar{\pi}_+:=\bar{\pi}_{\vert \Sgbar^+}$.
It is shown on \cite{Far_even} that the branch locus  of $\bar{\pi}_+$ is the divisor $\Delta_0\subset \MDM$ and the ramification divisor $\beta_0$, which fits the formula 
 $$K_{\Sgbar^+}=\bar{\pi}_+^*(K_{\MDM})+\beta_0,$$
 is the class of the divisor in $\Pic(\Sgpbar)$ whose general point $[(X,\theta,\beta)]$ satisfies the following:
\begin{itemize}
\item $X=C\cup E$ where is $C$ a general smooth curve of genus $g-1$ and $E$ is a $\bP^1$ meeting $C$ transversely in two distinct points $p$ and $q$;
\item $\theta|_{C}$ is an even theta-characteristic on $C$.
\end{itemize}

We are able to prove the following:

\begin{theorem}\label{sgplus}
For $g\geq 5$,  $S_g^{+o}:=\bar{\pi}_+^{-1}(\Mgo)\subset \cS_g$ has no holomorphic $1$-forms, that is  $H^0(S_g^{+o},\Omega^1_{S_g^{+o}})=0$.
\end{theorem}

\begin{proof} Since $\bar{\pi}_+: \Sgbar ^+\to \MDM$ is unramified over $\Delta_1$ (see \cite{Far_even}), in order to prove the theorem it is enough to prove that $H^1(S_g^{+o}, \bC)=0$ so that the assumptions of Theorem \ref{THM:MAIN} are satisfied.

Let us prove that $H^1(S_g^{+o}, \bC)=0$. It is known, by \cite{Har}, that $H^1(\cS_g, \bC)=0$ and so it is $H^1(\cS^+_g, \bC)=0$. But now $\Sgpbar$ is normal by \cite{Cor} and so $\Sing(\cS_g^+)$ has codimension at least $2$. Then, the codimension of $\cS_g^+\setminus S_g^{+o}$ is at most $2$ as well, and we have an isomorphism $\Pi_1(\hat{U})\simeq \Pi_1(\cS_g^+)$ which induces an isomorphism $H^1(S_g^{+o},\bC)\simeq H^1(\cS_g^+,\bC)=0$.
\end{proof}
 
\begin{remark} Some final remarks are in order.
\begin{enumerate}
    \item [a)] In the proof of the Theorem \ref{sgplus} we have used a result of Harer. We could instead use Corollary \ref{COR:BRANCHONDELTA0} to get the same result.
    \item [b)] A similar argument, using this time Corollary \ref{COR:BRANCHONDELTA0}, shows that, for $g\geq 5$, also $\cS^{-o}=\bar{\pi}_+^{-1}(\Mgo)\subset \cS_g^-$ has no holomorphic $1$-forms, that is  $H^0(\cS^{-o},\Omega^1_{\cS^{-o}})=0$. The key point is again that the map $\bar {\pi}_{-}$ is unramified over $\Delta_1$ (see subsection (1.2) in \cite{FV_odd}).
\end{enumerate}
\end{remark}



\end{document}